\documentclass[review]{elsarticle}

\usepackage{lineno,hyperref}
\usepackage[ruled,linesnumbered]{algorithm2e}
\usepackage{algcompatible}
\usepackage{enumitem}
\usepackage{epstopdf}
\usepackage{dsfont}
\usepackage{graphicx}
\usepackage{multirow}
\usepackage{hyperref}
\usepackage{amsthm}
\usepackage{amsfonts}
\usepackage{xcolor}
\usepackage{natbib} 
\usepackage{cool}
\usepackage[T1]{fontenc}

\newtheorem{theorem}{Theorem}
\newtheorem{remark}{Remark}
\newtheorem{corollary}{Corollary}
\newtheorem{lemma}{Lemma}

\newcommand{\R}{{\mathbb{R}}}
\newcommand{\N}{{\mathbb{N}}}

\modulolinenumbers[5]

\journal{Journal of \LaTeX\ Templates}

\begin{document}

\begin{frontmatter}

\title{Consistent Estimation of Residual Variance with Random Forest Out-Of-Bag Errors.}

\author{Burim Ramosaj$^*$, Markus Pauly}
\address{Institute of Statistics, Ulm University \\
	Helmholtzstrasse 20, 89075 Ulm, GERMANY}


\cortext[mycorrespondingauthor]{\textit{Corresponding Author:} Burim Ramosaj \\
	\textit{Email address:} \texttt{burim.ramosaj@uni-ulm.de} }


\begin{abstract}
	The issue of estimating residual variance in regression models has experienced relatively little attention in the machine learning community. However, the estimate is of primary interest in many practical applications, e.g. as a primary step towards the construction of prediction intervals. Here, we consider this issue for the random forest. 
 Therein, the functional relationship between covariates and response variable is modeled by a weighted sum of the latter. The dependence structure is, however, involved in the weights that are constructed during the tree construction process making the model complex in mathematical analysis. Restricting to $L_2$-consistent random forest models, we provide random forest based residual variance estimators and prove their consistency.
\end{abstract}


\end{frontmatter}

\linenumbers

\section{Introduction}

Random forest models are non-parametric regression resp. classification trees that highly rely on the idea of bagging and feature sub-spacing during tree construction. This way, one aims to construct highly predictive models by averaging (for continuous outcomes) or taking majority votes (for categorical outcomes)  over CART trees constructed on bootstrapped samples. At each node of a tree, the best cut is selected by optimizing a CART-split criterion such as the Gini impurity (for classification) or the squared prediction error (for regression) over a subsample of the feature space. This methodology has been proven to work well in predicting new outcomes as first shown in \cite{breiman2001random}. Despite that and closely related to the prediction of a new instance is the question how reliable this prognosis is. For example in \cite{khalilia2011predicting}, random forest models have been used in predicting disease risk in highly imbalanced data. Beyond point estimators, however, little information was known about the dispersion of disease risk prediction. 
In fact, estimating residual variance based upon machine learning techniques has experienced less attention compared to the extensive investigations on pure prediction.
One exception is given in \cite{mendez2011estimating}, where bootstrap corrected residual variance estimators are proposed. Moreover, they are analyzed in a simulation study for regression problems but no theoretical guarantees such as consistency have been proven. A similar observation holds for the jackknife-type sampling variance estimators given in \cite{wager2014confidence}. 
In the present paper we will close this gap by investigating the theoretical properties of a new residual variance estimator within the random forest framework. The estimator is inspired by the one proposed in \cite{mendez2011estimating} and is shown to be consistent for estimating residual variance if the random forest estimate for the regression function is $L_2$-consistent. 
Thereby, our theoretical derivations are build upon existing results in the literature. 

First theoretical properties of the random forest method such as ($L_2$-) consistency have already been proven in \cite{breiman2001random} while connections to layered nearest neighbors were made in \cite{lin2006random} and \cite{biau2010layered}. The early consistency results were later extended by several authors \citep{meinshausen2006quantile, biau2008consistency, wager2014confidence, scornet2015consistency, scornet2016asymptotics}; particularly allowing for stronger results (as central limit theorems) or 
a more reasonable mathematical model that better approximates the true random forest approach. In particular, varying mathematical forces such as feature sub-spacing, bagging and the tree construction process make the analysis of the true random forest as applied in practice very complicated. 

In the current work, we therefore decided to build upon the mathematical description of the random forest method as described in \cite{scornet2015consistency}. This allows the applicability of our estimator for a wide range of functional relationships while also incorporating relevant features of the algorithm such as the split-criterion.

Our paper is structured as follows. In the next section, we give a brief overview of the random forest and state the model framework. In addition, consistency results are stated. In the third section, we provide a residual variance estimate and prove its consistency in $L_1$-sense. Furthermore, bias-corrected residual variance estimators are proposed. Note that all proofs can be found in the appendix. 

\section{Model Framework and Random Forest}

Our framework is the $L_2$ regression estimation in which the covariable vector $\mathbf{X}$ is assumed to lie on the $p$-dimensional unit-cube, i.e. $\mathbf{X} \in [0,1]^p$. Of primary interest in the current paper is the estimation of the residual variance $\sigma^2$ in a functional relation of the form 
\begin{align} \label{RegModel}
	Y &= m(\mathbf{X}) + \epsilon. 
\end{align}

Here, $\mathbb{E}[|m(\mathbf{X}) |^2] < \infty$, $\mathbb{E}[\epsilon] = 0$ and $Var(\epsilon) \equiv \sigma^2 < \infty$ with $m: \R^p \rightarrow  \R$ and $\epsilon$ is independent of $\mathbf{X}$. Given a training set  
 \begin{align} \label{Sample}
 	\mathcal{D}_n =  \{(\mathbf{X}_i^\top, Y_i) \in [0,1]^p \times \R : i = 1, \dots, n \},
 \end{align}
of i.i.d. pairs $(\mathbf{X}_i, Y_i)$, $i = 1, \dots, n$, we aim to deliver an estimate $\hat{\sigma}_n^2$  that is at least $L_1$-consistent. The construction of $\hat{\sigma}_n^2$ will be based on the random forest estimate $m_n: [0,1]^p \rightarrow \R$ approximating the regression function $m$.  In the sequel, we will stick to the notation as given in \cite{scornet2015consistency} and shortly introduce the random forest model and corresponding mathematical forces involved in it. \\
The random forest model for regression is a collection of $M \in \N$ regression trees, where for each tree, a bootstrap sample is taken from $\mathcal{D}_n$ using with or without replacement procedures. This is denoted as the resampling strategy $\mathcal{S}$. Other sampling strategies than these two within the random forest model have been considered in \cite{ramosaj2017wins}, for example. Furthermore, at each node of the tree, feature sub-spacing is conducted selecting $m_{try} \in \{ 1, \dots ,p \}$ features for possible split direction. Denote with $\Theta$ the generic random variable responsible for both, the bootstrap sample construction and the feature sub-spacing procedure. Then, $\Theta_1, \dots, \Theta_M$ are assumed to be independent copies of $\Theta$ responsible for this random process in the $j$-th tree, independent of $\mathcal{D}_n$. The combination of the trees is conducted through averaging. i.e. 
\begin{align} \label{FiniteForest}
m_{M,n} (\mathbf{x}; \Theta_1, \dots \Theta_M, \mathcal{D}_n) = \frac{1}{M} \sum\limits_{j = 1}^M m_n(\mathbf{x}; \Theta_j, \mathcal{D}_n)
\end{align}

and is referred to as the finite forest estimate of $m$. As explained in \cite{scornet2015consistency}, the strong law of large numbers (for $M \rightarrow \infty$) allows to study $\mathbb{E}_\Theta[ m_n(\mathbf{x}; \Theta, \mathcal{D}_n) ]$ instead of $(\ref{FiniteForest})$. Hence, we set
\begin{align}\tag{3'}\label{InfForest}
	m_n(\mathbf{x}) = m_n(\mathbf{x}; \mathcal{D}_n) = \mathbb{E}_\Theta[ m_n(\mathbf{x}; \Theta, \mathcal{D}_n) ].
\end{align}
Similar to \cite{scornet2015consistency}, we refer to the random forest algorithm by identfiying three parameters responsibly for the random forest tree construction: 
\begin{itemize}
	\item $m_{try} \in \{1, \dots, p\}$ the number of pre-selected directions for splitting, 
	\item $a_n \in \{1, \dots, n \}$, the number of sampled points in the bootstrap step and
	\item $t_n \in \{ 1, \dots, a_n \}$, the number of leaves in each tree. 
\end{itemize}

Let $\{A_{\ell}^{(k)} \}_{\ell= 1}^{2^{k-1}} $ be a sequence of generic cells in $\R^p$ obtained at tree depth $k \in \N$, $k \le \lceil \log_2(t_n) \rceil + 1$ and denote by $N_n(A_\ell^{(k)})$ the number of observations falling in $A_\ell^{(k)}$, where we set $A_1^{(1)} = [0,1]^p$. Here, we denote a cut as the pair $(j, z) \in \{1, \dots, p \} \times [0,1]$, where $j$ represents the selected variable in which its domain is cut at $z$. Furthermore, let $\mathcal{C}_{A_\ell^{(k)}}$ be the set of all possible cuts in $A_\ell^{(k)}$. It should be noted that the restriction of the feature domain to the $p$-dimensional unit-cube $[0,1]^p$ is no restriction since the random forest is invariant under monotone transformations. 

Then formally, the random forest algorithm constructs decision trees resulting in regression estimators according to the following algorithm: \\

\begin{algorithm}[H] \label{Algorithm}
		\SetAlgoLined
	 \KwIn{Training set $\mathcal{D}_n$, number of decision trees $M$, $m_{try} \in \{1, \dots, p\}$, $a_n \in  \{1, \dots, p\}$, $t_n \in \{ 1, \dots, a_n \}$}
	\KwOut{Random forest estimate $m_{M,n}$}
	\For{ $j= 1, \dots, M$}{
		Select $a_n$ data points according to the resampling strategy $\mathcal{S}$ from $\mathcal{D}_n$; \\
		\While{$n_{nodes} \le t_n$}{
			Select without replacement a subset $\mathcal{M}_{try} \subseteq  \{1, \dots, p \}$ with $ |\mathcal{M}_{try}| = m_{try}$; \\ 
				\For{$ \ell = 1, \dots, k = \log_2(n_{nodes}) + 1$ }{
						   Find $(j_\ell^{*}, z_\ell^{*}) =  \arg \min\limits_{ \substack{j \in \mathcal{M}_{try} \\ (j, z) \in \mathcal{C}_{A_\ell^{(k )} }  } } L_n(j,z)$, where 
						    \begin{align*}
									    L_n(j, z ) &= \frac{1}{N_n(A_\ell^{(k)})} \sum\limits_{i = 1}^n (Y_i - \bar{Y}_{A_\ell^{(k)}})^2 \mathds{1}\{ \mathbf{X}_i \in A_\ell^{(k)} \}  \\
													     &- \frac{1}{N_n(A_\ell^{(k)})} \sum\limits_{ i = 1}^n (Y_i - \bar{Y}_{A_{\ell, L}^{(k)}} \mathds{1}\{X_{ji} < z \}  - \bar{Y}_{A_{\ell, R}^{(k) }} \mathds{1}\{X_{ji} \ge z\}  )^2 \mathds{1}\{\mathbf{X}_i \in A_{\ell}^{(k)} \} 
						    \end{align*}
									   is the $L_2$ regression cut criterion with $A_{\ell, L}^{(k)} = \{ \mathbf{x} \in A_ {\ell}^{(k)} | x_{j} < z \}$, $A_{\ell, R}^{(k)} = \{ \mathbf{x} \in A_{\ell}^{(k)} | x_j \ge z \}$ and $\bar{Y}_{ A_{\ell}^{(k)} }$ denotes the  mean of the $Y_i$'s over $A_{\ell}^{(k)}$ \;
									    Cut the cell $A_\ell^{(k)}$ at $(j_\ell^*, z_\ell^*)$ resulting into $A_{\ell,L}^{(k)}$ and $A_{\ell,R}^{(k)}$ \;
									    $n_{nodes} = n_{nodes} + 1$ \;
				}
		}
		Set $m_n(\cdot; \Theta_j, \mathcal{D}_n) $ as the $j$-th constructed tree.
	}
	\KwResult{Collection of $M$ decision trees $\{ m_n(\cdot; \Theta_j, \mathcal{D}_n) \}_{j = 1}^M $ used to obtain the aggregate regression estimate $m_{M,n}$ in (\ref{FiniteForest}) }
\caption{Random Forest $L_2$ Regression estimate.}
\end{algorithm}
\newpage

In order to establish $L_1$-consistency of the residual variance estimate $\hat{\sigma}_n^2$, we require at least $L_2$-consistency of the random forest method. That is, 
\begin{align}\label{ConsistencyRandomForest}
	\lim\limits_{n \rightarrow \infty} \mathbb{E}[(m_n(\mathbf{X}) - m(\mathbf{X}))^2] =  0,
\end{align}
where the expectation is taken with respect to $\mathbf{X}$ and $\mathcal{D}_n$. Here, $(\mathbf{X}, Y)$ is an independent copy of $(\mathbf{X}_i, Y_i)$ for $i \in \{1, \dots, n\}$. 

Several authors attempted to prove that (\ref{ConsistencyRandomForest}) is valid, i.e. that random forests are consistent in $L_2$-sense. \cite{biau2008consistency}, for example, assumed a simplified version of the random forest assuming that cuts happen independent of the response variable $Y$ in a purely random fashion. \cite{scornet2015consistency}, established consistency of the original random forest by assuming that $m$ is the additive expansion of continuous functions on the unit cube. Therein, proofs have been provided for fully grown trees ($t_n = a_n$) and not fully grown trees $(t_n < a_n)$ making additional assumptions on the asymptotic relation between $t_n$ and $a_n$. For example, Theorem~1 in  \cite{scornet2015consistency} guarantees condition (\ref{ConsistencyRandomForest}) for additive Gaussian regression models provided that $\mathbf{X}_i \stackrel{iid}{\sim} Unif([0,1]^p)$ and $a_n \rightarrow \infty$, $t_n \rightarrow \infty$, $t_n (\log(a_n))^9/a_n \rightarrow 0$ such that the resampling strategy $\mathcal{S}$ is restricted to sampling without replacement. 
In this context it should be noted that assumption (\ref{ConsistencyRandomForest}) does not automatically lead to pointwise consistency, since the latter is rather hard to prove for random forest models and counterexamples exist on the original random forest model as mentioned in \cite{wager2014asymptotic}. 

Anyhow, predicting outcomes among the training set $\mathcal{D}_n$ using the random forest is usually done by using \textit{Out-Of-Bag} (OOB) subsamples. That is, averaging does not happen over all $M$ trees but over those trees that did not have the corresponding data point in their resampled data set during tree construction. This way, one aims to deliver \textit{unbiased} estimators for predicted values. In addition, OOB samples have the advantage of delivering internal accuracy estimates, without separating the sample $\mathcal{D}_n$ into a training and test set. This way, the training sample size can be left sufficiently large. From a mathematical perspective, OOB-estimators of random forest have the nice property that independence between observed responses $Y_i$ and predicted $\hat{Y}_i$ remains valid for $i \in \{1, \dots, n\}$. This, because the prediction $\hat{Y}_i$ of $Y_i$ is based on samples not containing the point $(\mathbf{X}_i^\top, Y_i)$ for fixed $i = 1, \dots, n$. Thus, the independence property directly results from the independence assumption given in (\ref{Sample}). However, the justification to analyze infinite forests instead of finite forests as in (\ref{FiniteForest}) is unclear for OOB-estimates, since one does not consider the average over $M$ decision trees, but rather a random subset of $\{1, \dots, M\}$, depending on the data point one aims to predict. If we denote with $\hat{Y}_i = m_n^{OOB}(\mathbf{X}_i)$ the OOB prediction of $\mathbf{X}_i$, for $i \in \{1, \dots, n\}$ and $m_{M, n}^{OOB}(\mathbf{X}_i; \Theta_1, \dots, \Theta_M)$ the corresponding finite forest estimate, then we provide our first result proving the justification of considering infinite forests even for OOB samples. 

\begin{lemma} \label{OOBSLLN}
	Under Model (\ref{RegModel}), OOB predictions of finite forests are consistent, that is for all $\mathbf{x} \in [0,1]^p$
	$$ m_{M,n}^{OOB}(\mathbf{x}; \Theta_1, \dots, \Theta_M) \longrightarrow m_n^{OOB}(\mathbf{x}), \quad \mathbb{P}_\Theta \text{ - a.s.} \quad \text{ as } M \rightarrow \infty.$$
\end{lemma}

The consistency assumption in (\ref{ConsistencyRandomForest}) implies the consistency of the corresponding OOB-estimate. That is:

\begin{corollary} \label{ConsistencyOOB}
	For every fixed $i \in \{1, \dots, n\}$ under Model (\ref{RegModel}) and assuming (\ref{ConsistencyRandomForest}), OOB-estimators $m_n^{OOB}$ based on random forests are $L_2$-consistent in the following sense 
	\begin{align} \label{Consistency_OOBRandomForest}
	\lim\limits_{n \rightarrow \infty} \mathbb{E}[ (m_n^{OOB}(\mathbf{X}_i) - m(\mathbf{X}_i) )^2] = 0. 
	\end{align}
\end{corollary}

These preliminary results allow the construction of a consistent residual variance estimator based on OOB samples.

 \section{Residual Variance Estimation}

We estimate the residuals based on OOB samples, i.e. we set for $i = 1, \dots, n$ 
\begin{align} \label{ResEstimate}
	\hat{\epsilon}_i = Y_i - m_n^{OOB}(\mathbf{X}_i) ,
\end{align}
which we denote as {\it OOB-estimated residuals}. Their sample variance 
\begin{align} \label{VarOOBEstimate}
	\hat{\sigma}_{RF}^2 = \frac{1}{n} \sum\limits_{i = 1}^n ( \hat{\epsilon}_i - \bar{\epsilon}_\cdot )^2
\end{align}
or {\it OOB-estimated residual variance} is our proposed estimator. Here, $\bar{\epsilon}_\cdot = \sum_{i=1}^n\hat{\epsilon}_i/n$ denotes the mean of $\{ \hat{\epsilon}_i \}_{i = 1}^n$. A similar estimator has been proposed in \cite{mendez2011estimating}, for which simulation studies on some functional relationships between $\mathbf{X}$ and $Y$ were considered for practical implementation. The next result guarantees asymptotic unbiasedness and consistency of $\hat{\sigma}_{RF}^2$ under Assumption (\ref{ConsistencyRandomForest}). 
\begin{theorem} \label{VarConsis}
	Assume regression Model (\ref{RegModel}) and that (\ref{ConsistencyRandomForest}) is valid. Then the residual variance estimate $\hat{\sigma}_{RF}^2$ given in \eqref{VarOOBEstimate} is asymptotically unbiased and $L_1$-consistent as $n\to\infty$, i.e. 
	$$\hat{\sigma}_{RF}^2 \stackrel{L_1}{\longrightarrow} \sigma^2 \quad \text{as } n \rightarrow \infty.$$
\end{theorem}

\begin{remark}[Key Assumptions and Other Machine Learning Techniques]\mbox{ }\\
(a) Beyond Assumption \eqref{ConsistencyRandomForest} the structure of the random forest is only used to prove \eqref{Consistency_OOBRandomForest} and to maintain that 
the error variables $\epsilon_i$ are independent from $m_n^{OOB}(\mathbf{X}_i)$ and to have $m_n^{OOB}(\mathbf{X}_i)\stackrel{d}{=} m_n^{OOB}(\mathbf{X}_j)$, both for all fixed $1 \leq i,j \leq n$. Thus, the results can be extended to all methods guaranteeing these assumptions. \\
(b) Moreover, carefully checking the proof of Theorem~\ref{VarConsis}, the independence of $\epsilon$ towards $\mathbf{X}$ can also be substituted by $\mathbb{E}[\epsilon | \mathbf{X}] = 0$, $Var(\epsilon | \mathbf{X}) \equiv \sigma^2$ while still maintaining the consistency result.\\
\end{remark}


\subsection{Bias-corrected Estimation}

As explained in \cite{mendez2011estimating}, the estimator (\ref{VarOOBEstimate}) may be biased for finite sample size $n$. To this end, \cite{mendez2011estimating} 
proposed a biased-corrected version of $\hat{\sigma}_{RF}^2$ via parametric bootstrapping. Their idea is as follows: Given the data $\mathcal{D}_n$ generate i.i.d. {\it parametric bootstrap residuals} $\epsilon_{i,b}^*, i=1,\dots, n, b=1,\dots,B,$ independent from $\Theta,\dots, \Theta_M$, $\mathcal{D}_n$, with mean $\mathbb{E}^*[\epsilon_{1,1}^*] = \mathbb{E}[\epsilon_{1,1}^*|\mathcal{D}_n] = 0$ and variance $Var^*(\epsilon_{1,1}^*) = Var(\epsilon_{1,1}^*|\mathcal{D}_n) = \hat{\sigma}_{RF}^2$ from a parametric distribution with finite second moment, e.g. the normal distribution. Then, a {\it bias-corrected bootstrap version} of $\hat{\sigma}_{RF}^2$ is given by 
\begin{align}\label{BootstrapVersion}
	\hat{\sigma}_{RFboot}^2 &= \hat{\sigma}_{RF}^2 - \frac{1}{nB} \sum\limits_{b = 1}^B \sum\limits_{i = 1}^n (m_{n,b}^{OOB}(\mathbf{X}_i) - m_n^{OOB}(\mathbf{X}_i))^2 \notag \\ 
	&=: \hat{\sigma}_{RF}^2 - \hat{R}_B(m_n).
\end{align}
Here, $m_{n, b}^{OOB}$ is the OOB-estimation of $m$ using the tree structure of $m_n^{OOB}$ and feeding it with the bootstrapped sample $\mathcal{D}_{n, b}^* := \{ (\mathbf{X}_i, Y_{i,b}^*) : i = 1, \dots, n \}$ in which terminal node values are substituted with corresponding $Y_{i,b}^*$'s where $Y_{i,b}^* = m_n^{OOB}(\mathbf{X}_i)  + \epsilon_{i,b}^*$.

 In the following, we provide two important results regarding the bias-corrected version of $ \hat{\sigma}_ {RF}^2$. In Theorem \ref{ConsistencyBoot}, we prove that the bias-corected estimator in (\ref{BootstrapVersion}) is $L_1$-consistent. This guarantees that the proposed bootstrapping scheme does not systematically inflate our estimate. However, $\hat{\sigma}_{RFboot}^2$ comes with additional computation costs. Therefore, in Theorem \ref{BootTheorem}, we provide an asymptotic lower bound which enables a fast, bias-corrected estimation of $\sigma^2$ for finite sample sizes.

\begin{theorem}\label{ConsistencyBoot}
	Consider the random forest based parametric bootstrapping scheme as described for the estimation of (\ref{BootstrapVersion}). Assume further that the resampling strategy $\mathcal{S}$ is restricted to sampling without replacement. Assuming Model (\ref{RegModel}) and condition (\ref{ConsistencyRandomForest}) with $a_{n}^2/n \longrightarrow 0$, as $n \rightarrow \infty$. Then $\hat{\sigma}_{RFboot}^2$ is asymptotically $L_1$-consistent, that is 
	$$ \hat{\sigma}_{RFboot}^2 \stackrel{L_1}{\longrightarrow} \sigma^2, \text{ as } n \rightarrow \infty.$$
\end{theorem}

\begin{theorem}\label{BootTheorem}
	Consider the parametric bootstrapping scheme as described for the estimate in (\ref{BootstrapVersion}). Then for the random forest model, the follwoing inequality holds almost surely conditional on $\mathcal{D}_n$ as $B \rightarrow \infty$
	\begin{align*}
		\hat{R}_B(m_n) \ge \frac{\hat{\sigma}_{RF}^2}{a_n^2}.
	\end{align*}
\end{theorem}

The result in Theorem \ref{BootTheorem} leads to a residual variance estimate that is computationally cheaper than the corresponding bootstrapped version, i.e. one can consider 
\begin{align}
	\hat{\sigma}_{RFfast}^2 = \hat{\sigma}_{RF}^2\left(1 - \frac{1}{a_n^2}\right)
\end{align}
instead of $\hat{\sigma}_{RFboot}^2$, while saving considerable memory and computational time costs. This will lead to $\hat{\sigma}_{RF}^2 \ge \hat{\sigma}_{RFfast}^2 \ge \hat{\sigma}_{RFboot}^2$ almost surely. \\


\section{Conclusion}

The random forest is known as a powerful tool in applied data analysis for classification, regression and variable selection \citep{liaw2002classification, lunetta2004screening, diaz2006gene, strobl2007bias, genuer2010variable, khalilia2011predicting}. Beyond its practical use, corresponding theoretical properties have been investigated under various conditions \citep{breiman2001random, biau2008consistency, biau2010layered, wager2014asymptotic, scornet2015consistency} covering topics such as the $L_2$-consistent estimation of the regression function. However, a comprehensive treatment on how to estimate corresponding dispersion parameters as the variance is almost not to be found in the literature. 

An exception is given by the residual variance estimators proposed and examined in simulations in \cite{mendez2011estimating}. In the present paper, we complement their analyses by theoretically investigating residual variance estimators in regression models. To this end, we first show that analyzing the infinite forest estimate is legitimate, even when switching to OOB samples. This allows us to prove  consistency of the OOB-errors' sample variance in the $L_1$-sense if the random forest regression function estimate is assumed to be $L_2$-consistent. In addition, we also give some theoretical insight on the bias corrected residual variance estimate for finite samples as proposed in \cite{mendez2011estimating}. 

As the structure of the random forest is only needed to maintain the independence property in OOB samples, the current approach is also valid for any method that provides $L_2$-consistent regression function estimates.


\newpage

\section{Appendix.}

In this section we state the proofs for Lemma \ref{OOBSLLN}, Corollary \ref{ConsistencyOOB}, Theorem \ref{VarConsis}, Theorem \ref{ConsistencyBoot} and Theorem \ref{BootTheorem}.

\begin{proof}[Proof of Lemma \ref{OOBSLLN}]
	Let $i \in \{1, \dots, n \}$ be fixed and $\mathbf{x} \in [0,1]^p$ be an arbitrary and fixed point in the unit cube. Consider $m_n(\mathbf{x}) = \mathbb{E}_\Theta[m_n(\mathbf{x}; \Theta, \mathcal{D}_n)]$. If we denote with $Z_i = Z_i(M)$ the number of the $M$ regression trees not containing the $i$-th observation, then it follows 
	
	$$ Z_i \sim Bin(M, p_n)  \hspace{0.2cm} \text{ where } \hspace{0.2cm}  p_n = \begin{cases}
	1 - a_n/n &\text{ for subsampling} \\
	(1 - 1/n)^n &\text{ for bootstrapping with replacement.}
	\end{cases}$$

	such that $p_n>0$. Since $Z_i = \sum\limits_{\ell = 1}^M B_\ell$, where $B_\ell \stackrel{iid}{\sim} Bernoulli(p_n)$, it follows by the strong law of large numbers for fixed $n$ that $ Z_i/M \stackrel{a.s.}{\rightarrow} \mathbb{E}[B_1] = p_n $ as $M \rightarrow \infty$. Hence $Z_i(M) \stackrel{a.s.}{\longrightarrow} \infty$, as $M \rightarrow \infty$ for fixed $p_n > 0$. For given $\mathcal{D}_n^{(i)} := \mathcal{D}_n \setminus \{ (\mathbf{X}_i, Y_i) \}$, this justifies the consideration of 
	$$\frac{1}{Z_i(M)} \sum\limits_{l =1}^{Z_{i}(M) } m(\mathbf{x}; \Theta_l, \mathcal{D}_{n-1}) \stackrel{a.s.}{\longrightarrow} \mathbb{E}_\Theta[m(\mathbf{x}; \Theta, \mathcal{D}_{n-1})] =: m_n^{OOB}(\mathbf{x}), \quad M \rightarrow \infty. $$
	
\end{proof}

\begin{proof}[Proof of Corollary \ref{ConsistencyOOB}]
	Let $i \in \{1, \dots, n\}$ be fixed. Define the reduced sample $\mathcal{D}_n^{-(i)} :=  \mathcal{D}_{n} \setminus \{ (\mathbf{X}_i^\top, Y_i) \}$ as the OOB-sample of $i$. Then it follows from the independence assumption in (\ref{Sample}), that $(\mathbf{X}_i^\top, Y_i)$ is independent of $\mathcal{D}_n^{-(i)}$. Hence, $(\mathbf{X}_i^\top, Y_i)$ can be treated as an independent copy of $\mathcal{D}_n^{-(i)}$. The result thus follows immediately from (\ref{ConsistencyRandomForest}).
\end{proof}

\begin{proof}[Proof of Theorem \ref{VarConsis}] 	
	Consider $\hat{\epsilon}_i = Y_i - m_n^{OOB}(\mathbf{X}_i)$ and $\hat{\sigma}_{RF}^2 = \frac{1}{n}\sum\limits_{i = 1}^n (\hat{\epsilon}_i^2 - \bar{\epsilon}_\cdot)^2 = \frac{1}{n} \sum\limits_{i = 1}^n \hat{\epsilon}_i^2 - \bar{\epsilon}_\cdot^2$ from \eqref{ResEstimate}--\eqref{VarOOBEstimate}.
	Using Corollary \ref{ConsistencyOOB} and independence of $\epsilon_i$ and $m_n^{OOB}(\mathbf{X}_i)$ for all $i = 1, \dots, n$ it follows that
	\begin{align*}
	\mathbb{E}[ \frac{1}{n}\sum\limits_{i = 1}^n \hat{\epsilon}_i^2 ] &= \frac{1}{n} \sum\limits_{ i =1 }^n \mathbb{E}[  \{ (Y_i - m(\mathbf{X}_i) ) +( m(\mathbf{X}_i) - m_n^{OOB}(\mathbf{X}_i) ) \}^2  ]  \\	&= \mathbb{E}[ (Y_1 - m(\mathbf{X}_1))^2]  + \frac{1}{n} \sum\limits_{i = 1}^n  \{ 2 \mathbb{E}[(Y_i - m(\mathbf{X}_i))(m(\mathbf{X}_i) -  m_n^{OOB}(\mathbf{X}_i) )]  + \\ & \hspace{0.6cm } \mathbb{E}[(m(\mathbf{X}_i) - m_n^{OOB}(\mathbf{X}_i) )^2 ]  \} \\
	&= \sigma^2 + \frac{1}{n} \sum\limits_{i =1}^n  \{  2( \mathbb{E}[ m(\mathbf{X}_i) \mathbb{E}[Y_i | \mathbf{X}_i] ]  -  \mathbb{E}[Y_i m_n^{OOB}(\mathbf{X}_i) ]  - \mathbb{E}[m(\mathbf{X}_i)^2]  +  \\
	& \hspace{1.5cm } \mathbb{E}[ m(\mathbf{X}_i )   \mathbb{E}[m_n^{OOB}(\mathbf{X}_i) | \mathbf{X}_i]  ]  ) + \mathbb{E}[ ( m(\mathbf{X}_i) - m_n^{OOB}(\mathbf{X}_i) )^2 ] \} \\
	&= \sigma^2 + \frac{1}{n} \sum\limits_{i =1}^n  \{  2( \mathbb{E}[ m(\mathbf{X}_i) \mathbb{E}[Y_i | \mathbf{X}_i] ]  -  \mathbb{E}[m(\mathbf{X}_i) m_n^{OOB}(\mathbf{X}_i) ]  - \mathbb{E}[\epsilon_i m_n^{OOB}(\mathbf{X}_i)] - \\ 
	& \hspace{1.5cm }\mathbb{E}[m(\mathbf{X}_i)^2]  +   \mathbb{E}[ m(\mathbf{X}_i )   \mathbb{E}[m_n^{OOB}(\mathbf{X}_i) | \mathbf{X}_i]  ]  ) + \mathbb{E}[ ( m(\mathbf{X}_i) - m_n^{OOB}(\mathbf{X}_i) )^2 ] \} \\
	&=\sigma^2 + \frac{1}{n} \sum\limits_{ i = 1}^n \{ 2( \mathbb{E}[m(\mathbf{X}_i)^2] -  \mathbb{E}[ m(\mathbf{X}_i)  \mathbb{E}[m_n^{OOB}(\mathbf{X}_i) | \mathbf{X}_i ] ]  - \mathbb{E}[m(\mathbf{X}_i)^2 ] + \\ &\hspace{1.5cm}\mathbb{E}[ m(\mathbf{X}_i) \mathbb{E}[ m_n^{OOB}(\mathbf{X}_i) | \mathbf{X}_i ] ] ) + \mathbb{E}[ ( m(\mathbf{X}_i) - m_n^{OOB}(\mathbf{X}_i)  )^2 ]     \} \\
	&= \sigma^2 + \frac{1}{n}\sum\limits_{i = 1}^n \mathbb{E}[ ( m(\mathbf{X}_i)  - m_n^{OOB}(\mathbf{X}_i) )^2 ] \\
	&= \sigma^2 + \mathbb{E}[ ( m(\mathbf{X}_1) - m_n^{OOB}(\mathbf{X}_1) )^2 ]  \longrightarrow \sigma^2
	\end{align*}
	by Corollary~\ref{ConsistencyOOB} as $n\to\infty$. 
	The second and last equality follows from the identical distribution of the sequences $ \{Y_i - m( \mathbf{X}_i) \}_{i = 1}^n$ resp. $\{ m(\mathbf{X}_i) - m_n^{OOB}(\mathbf{X}_i) \}_{i = 1}^n$.

	Furthermore, let $\varDelta_n(\mathbf{X}_i) := m(\mathbf{X}_i) - m_n^{OOB}(\mathbf{X}_i) $. Then using the Cauchy-Schwarz inequality we obtain
	\begin{align*}
	0 \le \mathbb{E}\left[ \left( \frac{1}{n} \sum\limits_{i = 1}^n \hat{\epsilon}_i \right)^2 \right] &= \frac{1}{n^2} \mathbb{E}\left[  \left( \sum\limits_{i = 1}^n  \varDelta_n(\mathbf{X}_i)  + \epsilon_i\right)^2 \right] \\
	&= \frac{1}{n^2} \mathbb{E}[ \sum\limits_{i = 1}^n (  \varDelta_n(\mathbf{X}_i) + \epsilon_i )^2  ] + \frac{1}{n^2} \mathbb{E}[ \sum\limits_{i \neq j} ( \varDelta_n(\mathbf{X}_i) + \epsilon_i  )(\varDelta_n(\mathbf{X}_j) + \epsilon_j    ) ] \\
	&= \frac{1}{n^2}  \sum\limits_{i = 1}^n \mathbb{E}[\varDelta_n(\mathbf{X}_i) ^2 + 2 \epsilon_i \varDelta_n(\mathbf{X}_i) + \epsilon_i^2 ] + \\
	& + \frac{1}{n^2}  \sum\limits_{i \neq j} \mathbb{E}[\varDelta_n(\mathbf{X}_i) \varDelta_n(\mathbf{X}_j)   + \varDelta_n(\mathbf{X}_i) \epsilon_j + \varDelta_n(\mathbf{X}_j) \epsilon_i ]\\
	&\le \frac{\sigma^2}{n} + \frac{1}{n^2} \sum\limits_{i = 1}^n \mathbb{E}[\mathbb\varDelta_n(\mathbf{X}_i)^2] + \frac{1}{n^2}\sum\limits_{i \neq j}  \sqrt{\mathbb{E}[ \varDelta_n(\mathbf{X}_i)^2 ]  \mathbb{E}[\varDelta_n(\mathbf{X}_j)^2 ] } + \\
	&+ \sigma \sqrt{ \mathbb{E}[\varDelta_n(\mathbf{X}_i)^2 ]} + \sigma \sqrt{ \mathbb{E}[  \varDelta_n(\mathbf{X}_j)^2 ] } \\
	&\stackrel{id}{=} \frac{\sigma^2}{n} + \frac{ \mathbb{E}[ \varDelta_n(\mathbf{X}_1)^2] }{n} + (1 - \frac{1}{n})( \mathbb{E}[ \varDelta_n(\mathbf{X}_1)^2  ] + 2 \sigma  \sqrt{\mathbb{E}[ \varDelta_n(\mathbf{X}_1)^2 ]} ) \\
	& \longrightarrow 0
	\end{align*}
	by Corollary~\ref{ConsistencyOOB} as $n\to\infty$ which completes the proof.

\end{proof}

\begin{proof}[Proof of Theorem \ref{ConsistencyBoot}]
 To be mathematically precise, let $\Theta,\dots, \Theta_M$ and $(Y_i,{\bf X}_i)$ be defined on some probability space $(\Omega,\mathcal A, \mathbb{P})$ and let the parametric bootstrap variables $\epsilon_{i,b}^*$ be defined on another  probability space $(\Omega^*,\mathcal A^*, \mathbb{P}^*)$. Then, all random variables can be defined (via projections) on the joined product space $(\Omega\times \Omega^*,\mathcal A \otimes \mathcal A^*, \mathbb{P} \otimes \mathbb{P}^*)$; explaining the assumption that the random variables $\epsilon_{i,b}^*$ are independent from $\mathcal{D}_n$ and i.i.d. generated from a distribution with finite second moment with $Var^*(\epsilon_{1,1}^*) = \hat{\sigma}_{RF}^2$ and $\mathbb{E}^*[\epsilon_{1,1}^*] = 0$. 
 
 Within this framework consider $Y_{i,b}^* = m_n^{OOB}(\mathbf{X}_i) + \epsilon_{i,b}^*$ and denote with $\mathcal{D}_{n, b}^* := \{ (\mathbf{X}_i, Y_{i,b}^*) : i = 1, \dots, n \}$ the set of the $b$-th bootstrapped sample for $b \in \{1, \dots, B\}$. Then the sequence of sets $\{\mathcal{D}_{n,b}^*\}_{b = 1}^B$ is independent. In particular, conditioned on $\mathcal{D}_n$, $\{ Y_{i,b}^* \}_{b}$ forms a sequence of i.i.d. random variables. 
 
 Now, note that random forest models are the weighted sum of the response variable. Hence, denoting with $A_n({\mathbf{x}; \Theta} )\subset [0,1]^p$ the hyper-rectangle obtained after constructing one random decision tree with seed parameter $\Theta$ containing $\mathbf{x}$, then the infinite random forest model can be rewritten as  
 \begin{align}
 	m_n^{OOB}(\mathbf{X}_i) &= \sum\limits_{j = 1}^n W_{n, j}^{OOB}(\mathbf{X}_i) Y_{j},
 \end{align}
  see, e.g., the proof of Theorem 2 in \cite{scornet2015consistency} for a similar observation. Here, $\sum\limits_{j = 1}^n W_{n,j}^{OOB}(\mathbf{X}_i) = 1$ holds almost surely and the weights $W_{n, j}^{OOB}$ are defined as 
  $$
  W_{n,j}^{OOB}(\mathbf{X}_i) = \mathbb{E}_\Theta \left[ \frac{\mathds{1}\{ \mathbf{X}_j \in A_n(\mathbf{X}_i; \Theta) \}  }{N_n( A_n(\mathbf{X}_i; \Theta ) )} \right],
  $$
  where $N_n(A_{n}(\mathbf{X}_i; \Theta))$ is the number of data points falling in $A_n(\mathbf{X}_i; \Theta)$. Further let $\mathbf{X}_j \stackrel{\Theta }{\leftrightarrow} \mathbf{X}_i$ be the event that both points, $\mathbf{X}_i$ and $\mathbf{X}_j$, fall in the same cell under the tree constructed by $\Theta$.
  Due to sampling without replacement, there are ${n-1}\choose{a_n-1} $ choices to pick a fixed observation $\mathbf{X}_i$. Therefore, we obtain 
  \begin{align}\label{UpperBoundWeight}
  	W_{nj}^{OOB}(\mathbf{X}_i) \le\max\limits_{1 \le i  \le n} \mathbb{P}_\Theta ( \mathbf{X}_j \stackrel{\Theta}{\leftrightarrow} \mathbf{X}_i ) \le \frac{ { n-2\choose a_{n-1} -1 }  }{ { n-1 \choose a_{n-1}   }  } \le \frac{a_{n-1}}{n-1}
  \end{align}

Setting $A_{n,i} =  \sum\limits_{j = 1}^n W_{nj}^{OOB}(\mathbf{X}_i) \cdot (m_n^{OOB}(\mathbf{X}_j)  - m(\mathbf{X}_j))$ and $B_{n,i} = \sum\limits_{j = 1}^n W_{nj}^{OOB}(\mathbf{X}_i) \cdot (\epsilon_j^* - \epsilon_j)$, we obtain the following result for every fixed $b \in \{1, \dots, B\}$ using the Cauchy-Schwarz inequality:
\begin{align}\label{BootstrapIneq}
\mathbb{E}[(m_{n,b}^{OOB}(\mathbf{X}_i) - m_n^{OOB}(\mathbf{X}_i) )^2]	&= \mathbb{E} \left[ \left(\sum\limits_{j = 1}^n W_{nj}^{OOB}(\mathbf{X}_i) Y_j^*  - m_n^{OOB}(\mathbf{X}_i)  \right)^2 \right] \notag \\ \notag
 &= \mathbb{E}\left[\left( \sum\limits_{j = 1}^n W_{nj}^{OOB}(\mathbf{X}_i) \cdot (Y_j^* - Y_j)  \right)^2  \right]\\ \notag
 &= \mathbb{E}[ (A_{n,i} + B_{n,i})^2 ] \\
 &\le \mathbb{E}[A_{n,i}^2] + 2 \mathbb{E}[A_{n,i}^2]\mathbb{E}[B_{n,i}^2] + \mathbb{E}[B_{n,i}^2]
\end{align} 
In order to prove $L_1$-consistency of the bootstrapped corrected estimate, based on (\ref{BootstrapIneq}), we only need to show that $\mathbb{E}[A_{n,i}^2] \rightarrow 0$ and $\mathbb{E}[ B_{n,i}^2 ] \rightarrow 0$ as $n \rightarrow \infty$. Now, note that $\mathbb{E}[ ( \epsilon_j^* - \epsilon_j )^2 | \mathcal{D}_n ] = \hat{\sigma}_{RF}^2 - 2 \epsilon_j \mathbb{E}[ \epsilon_j^* | \mathcal{D}_n ] + \epsilon_j^2 = \hat{\sigma}_{RF}^2 + \epsilon_j^2 $ almost surely. Conditioning on $\mathcal{D}_n$, we know that $(\epsilon_j^* - \epsilon_j)$ and $(\epsilon_\ell^* - \epsilon_\ell)$ are independent for $j \neq \ell$ such that $\mathbb{E}[ (\epsilon_j^* - \epsilon_j)(\epsilon_\ell^* - \epsilon_\ell) | \mathcal{D}_n ]  = \mathbb{E}[ (\epsilon_j^* - \epsilon_j) | \mathcal{D}_n ] \mathbb{E}[ (\epsilon_\ell^* - \epsilon_\ell)| \mathcal{D}_n ]   = \epsilon_j \epsilon_\ell$ almost surely. Combining these two results, we obtain with (\ref{UpperBoundWeight}): 
\begin{align}
	\mathbb{E}[ B_{n,i}^2 ] &= \sum\limits_{j = 1}^n\mathbb{E}[ W_{nj}^{OOB}(\mathbf{X}_i)^2 (\epsilon_j^* - \epsilon_j)^2] + \notag \\
	& + \sum\limits_{j \neq \ell} \mathbb{E}[W_{nj}^{OOB}(\mathbf{X}_i) W_{n\ell}^{OOB}(\mathbf{X}_i)(\epsilon_j^* - \epsilon_j)(\epsilon_\ell^* - \epsilon_\ell)  ] \notag \\
	&= \sum\limits_{j = 1}^n\mathbb{E}[ W_{nj}^{OOB}(\mathbf{X}_i)^2  \mathbb{E} [(\epsilon_j^* - \epsilon_j)^2 | \mathcal{D}_n] ]  + \notag \\
	& + \sum\limits_{j \neq \ell} \mathbb{E}[W_{nj}^{OOB}(\mathbf{X}_i) W_{n\ell}^{OOB}(\mathbf{X}_i) \mathbb{E}[ (\epsilon_j^* - \epsilon_j)(\epsilon_\ell^* - \epsilon_\ell) | \mathcal{D}_n ]  ] \\ &= \sum\limits_{ j = 1}^n \mathbb{E}[ W_{n,j}^{OOB}(\mathbf{X}_i)^2( \hat{\sigma}_{RF}^2 + \epsilon_j^2 ) ] + \notag \\
	&+ \sum\limits_{j \neq \ell} \mathbb{E}[ W_{nj}^{OOB}(\mathbf{X}_i) W_{n\ell}^{OOB}(\mathbf{X}_i) \epsilon_j \epsilon_\ell ]  \notag  \\
	&\le \frac{a_{n-1}^2}{n-1} \frac{n}{n-1} (\mathbb{E}[\hat{\sigma}_{RF}^2]+ \sigma^2) +  \frac{a_{n-1}^2}{(n-1)^2} \sum\limits_{j \neq \ell} \mathbb{E}[\epsilon_j \epsilon_\ell] \notag \\
	&= \frac{a_{n-1}^2}{n-1} \frac{n}{n-1} (\mathbb{E}[\hat{\sigma}_{RF}^2]+ \sigma^2)  \longrightarrow 0, \quad n \rightarrow \infty,  \label{BBoot}
\end{align}
where the inequality results by applying (\ref{UpperBoundWeight}) on the weights and the last equality from the fact that $\mathbb{E}[ \epsilon_j \epsilon_\ell ] = \mathbb{E}[\epsilon_j] \mathbb{E}[\epsilon_\ell] = 0$, the convergence from Theorem \ref{VarConsis} and $a_n^2/n \rightarrow 0$.
Furthermore, by applying Jensen's inequality, we obtain 
\begin{align}
	\mathbb{E}[A_{n,i}^2] &\le \mathbb{E}[ \sum\limits_{j = 1}^n W_{nj}^{OOB}(\mathbf{X}_i) ( m_n^{OOB}(\mathbf{X}_j) - m(\mathbf{X}_j) )^2 ] \notag \\
	&= \mathbb{E}[ \sum\limits_{ \substack{j=1 \\ j\neq i}}^n W_{nj}^{OOB}(\mathbf{X}_i) ( m_n^{OOB}(\mathbf{X}_j) - m( \mathbf{X}_j) )^2 ]  + \notag \\
	&+ \mathbb{E}[  \mathbb{E}_\Theta[ N_n(A_n(\mathbf{X}_i))^{-1} ]  (m_n^{OOB}(\mathbf{X}_i) - m(\mathbf{X}_i))^2 ]  \notag \\
	& \le \mathbb{E}[ ( m_n^{OOB}(\mathbf{X}_1)  - m(\mathbf{X}_1))^2  \sum\limits_{ \substack{j=1 \\ j\neq i} } W_{nj}^{OOB}(\mathbf{X}_i) ] + \mathbb{E}[ (m_n^{OOB}(\mathbf{X}_1) - m(\mathbf{X}_1)^2 ] \notag \\
	&\le 2 \mathbb{E}[ (m_n^{OOB}(\mathbf{X}_1) - m(\mathbf{X}_1))^2] \longrightarrow 0, \quad n \rightarrow \infty, \label{ABoot}
\end{align}

where the last two inequalities arise by using the identical distribution of the sequence $\{ (m_n^{OOB}(\mathbf{X}_i) - m(\mathbf{X}_i))^2 \}_{i = 1}^n$ and the fact that the random forest weights sum up to one.  Finally, we have
\begin{align}
	\mathbb{E}[ \hat{R}_B(m_n) ] &= \frac{1}{n} \sum\limits_{i = 1}^n \mathbb{E}[ \frac{1}{B} \sum\limits_{b = 1}^B \mathbb{E}[ (m_{n,b}^{OOB}(\mathbf{X}_i) - m_n^{OOB}(\mathbf{X}_i))^2 | \mathcal{D}_n ]   ] \notag \\
	&=\frac{1}{n} \sum\limits_{i = 1}^n \mathbb{E}[ \mathbb{E}[ ( m_{n,1}^{OOB}(\mathbf{X}_i ) - m_n(\mathbf{X}_i) )^2 | \mathcal{D}_n ] ] \notag \\
	&= \mathbb{E}[ ( m_{n,1}^{OOB}(\mathbf{X}_i ) - m_n(\mathbf{X}_i) )^2 ]  \longrightarrow 0, \label{FinalResBoot}
 \end{align}
 where the last two equalities follow from the identical distribution of $\{ m_{n,b}(\mathbf{X}_i) \}_{b = 1}^B$ with respect to the bootstrap measure $\mathbb{P}^*$ and the identical distribution of $\{  \mathbb{E}[ ( m_{n,1}^{OOB}(\mathbf{X}_i ) - m_n(\mathbf{X}_i) )^2 | \mathcal{D}_n ]\}_{ i = 1}^n$.
The convergence in (\ref{FinalResBoot}) follows by plugging in (\ref{BBoot}) and (\ref{ABoot}) into (\ref{BootstrapIneq}).

\end{proof}

\begin{proof}[Proof of Theorem \ref{BootTheorem}]
   Sticking to the same notation as in Theorem \ref{ConsistencyBoot}, we obtain the following lower bound for the random forest weights: 
  \begin{align}\label{WeightInequality}
  	W_{n,j}(\mathbf{X}_i)  \ge \mathbb{E}_\Theta[ \mathds{1} \{ \mathbf{X}_j \stackrel{\Theta}{\leftrightarrow} \mathbf{X}_i \}(a_n)^{-1} ] = \frac{\mathbb{P}_\Theta ( \mathbf{X}_j \stackrel{\Theta}{\leftrightarrow} \mathbf{X}_i )}{a_n}.
  \end{align}
  
  Since the prescribed parametric bootstrap approach makes use of the same tree structure as $m_n^{OOB}$, the bootstrapped prediction $\hat{Y}_{i,b}^* = m_{n,b}^{OOB}(\mathbf{X}_i)$ can also be rewritten as
\begin{align}
	m_{n,b}^{OOB}(\mathbf{X}_i) = \sum\limits_{j = 1}^n W_{n,j}^{OOB}(\mathbf{X}_i) \cdot  Y_{j,b}^\star \quad b = 1, \dots, B
\end{align}
 As these quantities are i.i.d. in the index $b$ for fixed $i$, we can apply the strong law of large numbers conditioned on fixed $\mathcal{D}_n$ to obtain
 \begin{align}\label{StrongLawBoot}
 	R_B(m_n) \longrightarrow \frac{1}{n} \sum\limits_{i = 1}^n \mathbb{E}^*\left[  (m_{n,1}(\mathbf{X}_i) -  m_n(\mathbf{X}_i))^2  \right] 
 \end{align}
 almost surely as $B \rightarrow \infty$ given $\mathcal{D}_n$. Moreover, we have
  \begin{align}\label{FirstMoment}
  	\mathbb{E}^*[m_{n,1}^{OOB}(\mathbf{X}_i)] &= \sum\limits_{j = 1}^n W_{n,j}^{OOB}(\mathbf{X}_i) \cdot \mathbb{E}^*[Y_{j,1}^*] \notag \\
  	&= \sum\limits_{j = 1}^n W_{n,j}^{OOB}(\mathbf{X}_i) \cdot m_n^{OOB}(\mathbf{X}_j).
  \end{align}
 Furthermore, due to the independence of the sequence $\{ Y_j^* \}_j$ with respect to the boostrap measure $\mathbb{P}^*$, we obtain 
  \begin{align} \label{SecondMoment}
  	\mathbb{E}^*[ m_{n,1}^{OOB}(\mathbf{X}_i)^2 ] &=  \sum\limits_{j = 1}^n( W_{n,j}^{OOB}(\mathbf{X}_i))^2 \cdot \mathbb{E}^*[(Y_j^*)^2] +  \\
  	&+ \sum\limits_{j \neq k} W_{n,j}^{OOB}(\mathbf{X}_i) W_{n,k}^{OOB}(\mathbf{X}_i) \mathbb{E}^*[Y_j^*] \mathbb{E}^*[Y_k^*] \notag\\
  	&= \sum\limits_{j = 1}^n (W_{n,j}^{OOB}(\mathbf{X}_i) )^2 ( \hat{\sigma}_{RF}^2 + m_n^{OOB}(\mathbf{X}_j) ^2) + \notag\\
  	&+ \sum\limits_{j \neq k} W_{n,j}^{OOB}(\mathbf{X}_i) \cdot W_{n,k}^{OOB}(\mathbf{X}_i) \cdot m_n^{OOB}(\mathbf{X}_j) \cdot m_n^{OOB}(\mathbf{X}_k) \notag\\
  	&= \hat{\sigma}_{RF}^2 \sum\limits_{j = 1}^n (W_{n,j}^{OOB}(\mathbf{X}_i))^2 + \left(  \sum\limits_{j = 1}^n W_{n,j}^{OOB}(\mathbf{X}_i) \cdot m_n^{OOB}(\mathbf{X}_j) \right)^2. \notag
  \end{align}
  Combining the results from (\ref{WeightInequality}), (\ref{FirstMoment}) and (\ref{SecondMoment}), we obtain: 
  \begin{align*}
  	\mathbb{E}^*[ ( m_{n,1}^{OOB}(\mathbf{X}_i) - m_n(\mathbf{X}_i) )^2 ] &= \hat{\sigma}_{RF}^2 \sum\limits_{j = 1}^n (W_{n,j}^{OOB}(\mathbf{X}_i))^2 +  \\
  	&+ \left( \sum\limits_{j = 1}^n W_{n,j}^{OOB}(\mathbf{X}_i)\cdot  m_n^{OOB}(\mathbf{X}_i) \right)^2 \\ 
  	&- 2 \left(  \sum\limits_{j = 1}^n W_{n,j}^{OOB}(\mathbf{X}_i) \cdot m_n^{OOB}(\mathbf{X}_j) \cdot m_n^{OOB}(\mathbf{X}_i)  \right) \\
  	&+ m_n^{OOB}(\mathbf{X}_i)^2 \\
  	&= \left( \sum\limits_{j = 1}^n W_{n,j}^{OOB}(\mathbf{X}_i) \cdot ( m_n^{OOB}(\mathbf{X}_j) - m_n^{OOB}(\mathbf{X}_i) ) \right)^2 \\
  	&+ \hat{\sigma}_{RF}^2 \sum\limits_{j = 1}^n (W_{n,j}^{OOB}(\mathbf{X}_i) )^2 \\
  	&\ge \frac{ \hat{\sigma}_{RF}^2}{a_n^2}  \sum\limits_{j = 1}^n \mathbb{P}_\Theta^2( \mathbf{X}_j \stackrel{\Theta}{\leftrightarrow}  \mathbf{X}_i  ) \ge \frac{ \hat{\sigma}_{RF}^2}{a_n^2}, 
  \end{align*}
  where we inserted $\sum_{\ell=1}^n W_{n,\ell}^{(OOB)}({\bf X}_i) =1$ (almost surely) in the third and second last step and utilized that 
  $\mathbb{P}_\Theta^2( \mathbf{X}_i \stackrel{\Theta}{\leftrightarrow}  \mathbf{X}_i) =1$ in the last inequality. Finally, the result follows from (\ref{StrongLawBoot}).
 
\end{proof}

\section*{References}

\bibliography{BibFile}

\begin{thebibliography}{10}
\expandafter\ifx\csname url\endcsname\relax
  \def\url#1{\texttt{#1}}\fi
\expandafter\ifx\csname urlprefix\endcsname\relax\def\urlprefix{URL }\fi
\expandafter\ifx\csname href\endcsname\relax
  \def\href#1#2{#2} \def\path#1{#1}\fi

\bibitem{breiman2001random}
L.~Breiman, Random {F}orests, Machine Learning 45~(1) (2001) 5--32.

\bibitem{khalilia2011predicting}
M.~Khalilia, S.~Chakraborty, M.~Popescu, Predicting disease risks from highly
  imbalanced data using {R}andom {F}orest, BMC Medical Informatics and Decision
  Making 11~(1) (2011) 51.

\bibitem{mendez2011estimating}
G.~Mendez, S.~Lohr, Estimating {R}esidual {V}ariance in {R}andom {F}orest
  {R}egression, Computational Statistics \& Data Analysis 55~(11) (2011)
  2937--2950.

\bibitem{wager2014confidence}
S.~Wager, T.~Hastie, B.~Efron, Confidence {I}ntervals for {R}andom {F}orests:
  The {J}ackknife and the {I}nfinitesimal {J}ackknife, Journal of Machine
  Learning Research 15~(1) (2014) 1625--1651.

\bibitem{lin2006random}
Y.~Lin, Y.~Jeon, Random {F}orests and {A}daptive {N}earest {N}eighbors, Journal
  of the American Statistical Association 101~(474) (2006) 578--590.

\bibitem{biau2010layered}
G.~Biau, L.~Devroye, On the {L}ayered {N}earest {N}eighbour {E}stimate, the
  bagged {N}earest {N}eighbour {E}stimate and the {R}andom {F}orest method in
  {R}egression and {C}lassification, Journal of Multivariate Analysis 101~(10)
  (2010) 2499--2518.

\bibitem{meinshausen2006quantile}
N.~Meinshausen, Quantile {R}egression {F}orests, Journal of Machine Learning
  Research 7~(6) (2006) 983--999.

\bibitem{biau2008consistency}
G.~Biau, L.~Devroye, G.~Lugosi, Consistency of random forests and other
  averaging classifiers, Journal of Machine Learning Research 9~(Sep) (2008)
  2015--2033.

\bibitem{scornet2015consistency}
E.~Scornet, G.~Biau, J.-P. Vert, Consistency of {R}andom {F}orests, The Annals
  of Statistics 43~(4) (2015) 1716--1741.

\bibitem{scornet2016asymptotics}
E.~Scornet, On the {A}symptotics of {R}andom {F}orests, Journal of Multivariate
  Analysis 146 (2016) 72--83.

\bibitem{ramosaj2017wins}
B.~Ramosaj, M.~Pauly, Predicting missing values: A comparative study on
  non-parametric approaches for imputation., arXiv preprint arXiv:1711.11394
  (2017).

\bibitem{wager2014asymptotic}
S.~Wager, S.~Athey, Estimation and inference of heterogeneous treatment effects
  using random forests, Journal of the American Statistical Association
  113~(523) (2018) 1228--1242.

\bibitem{liaw2002classification}
A.~Liaw, M.~Wiener, et~al., Classification and {R}egression by randomforest, R
  News 2~(3) (2002) 18--22.

\bibitem{lunetta2004screening}
K.~L. Lunetta, L.~B. Hayward, J.~Segal, P.~Van~Eerdewegh, Screening large-scale
  association study data: {E}xploiting interactions using {R}andom {F}orests,
  BMC Genetics 5~(1) (2004) 32.

\bibitem{diaz2006gene}
R.~D{\'\i}az-Uriarte, S.~A. De~Andres, Gene {S}election and {C}lassification of
  {M}icroarray {D}ata using {R}andom {F}orest, BMC Bioinformatics 7~(1) (2006)
  3.

\bibitem{strobl2007bias}
C.~Strobl, A.-L. Boulesteix, A.~Zeileis, T.~Hothorn, Bias in {R}andom {F}orest
  variable importance measures: {I}llustrations, sources and a solution, BMC
  Bioinformatics 8~(1) (2007) 25.

\bibitem{genuer2010variable}
R.~Genuer, J.-M. Poggi, C.~Tuleau-Malot, Variable {S}election using {R}andom
  {F}orests, Pattern Recognition Letters 31~(14) (2010) 2225--2236.

\end{thebibliography}

\end{document}